\documentclass[11pt]{amsart}

\usepackage{amssymb,amsmath,epsfig,graphics}

\numberwithin{equation}{section}
\newcommand{\beq}{\begin{equation}}
\newcommand{\eeq}{\end{equation}}
\newcommand{\bea}{\begin{eqnarray}}
\newcommand{\eea}{\end{eqnarray}}
\newcommand{\beas}{\begin{eqnarray*}}
\newcommand{\eeas}{\end{eqnarray*}}

\newcommand{\Ga}{\Gamma}
\newcommand{\M}{\mathbb M}
\newcommand{\1}{\mathbf 1}
\newcommand{\Ent}{\textrm{Ent}}
\newcommand{\W}{\mathcal W}
\newtheorem{theorem}{Theorem}[section]
\newtheorem{definition}[theorem]{Definition}
\newtheorem{proposition}[theorem]{Proposition}
\newtheorem{corollary}[theorem]{Corollary}
\newtheorem{lemma}[theorem]{Lemma}
\newtheorem{remark}[theorem]{Remark}
\newtheorem{example}[theorem]{Example}
\newtheorem{examples}[theorem]{Examples}

\newtheorem{foo}[theorem]{Remarks}

\newcommand{\bM}{\mathbb M}

\newcommand{\F}{\mathcal F}

\newcommand{\R}{\mathbb R}

\title[Log-sobolev inequalities]{Log-Sobolev inequalities for subelliptic operators satisfying a generalized curvature dimension inequality}

\author{Fabrice Baudoin}
\address{Department of Mathematics\\Purdue University \\
West Lafayette, IN 47907 \\
USA
} \email[Fabrice Baudoin]{fbaudoin@math.purdue.edu}
\thanks{First author supported in part by
NSF Grant DMS 0907326}

\author{Michel Bonnefont}
\address{Institut de Mathmatiques de Bordeaux \\
Universit\'e Bordeaux 1 \\
33405 Talence \\
FRANCE 
} \email[Michel Bonnefont]{michel.bonnefont@math.u-bordeaux1.fr}

\begin{document}

\maketitle

\begin{abstract}
Let $\M$ be a smooth connected manifold endowed with a smooth measure $\mu$ and a smooth locally subelliptic diffusion operator $L$ which is symmetric with respect to $\mu$. We assume that $L$ satisfies a  generalized curvature dimension inequality as introduced by Baudoin-Garofalo \cite{BG1}. Our goal is to discuss functional inequalities for $\mu$ like the Poincar\'e inequality, the log-Sobolev inequality or the Gaussian logarithmic isoperimetric inequality.  
\end{abstract}

\tableofcontents

\section{Introduction, main results and examples}

Logarithmic Sobolev inequalities, introduced and studied by L. Gross \cite{G},  are a major tool for the analysis of finite or infinite dimensional spaces, see for instance  \cite{SOB} and the references therein. The celebrated  Bakry-\'Emery criterion \cite{Bakry-Emery} which is based on the so-called $\Gamma_2$ calculus for diffusion operators provides a powerful way to establish such inequalities. However this criterion requires some ellipticity property from the diffusion operator and fails to hold even for simple subelliptic  diffusion operators like the sub-Laplacian on the Heisenberg group (see \cite{J}). However in the past few years, numerous works like \cite{BBBC}, \cite{BBBQ},  \cite{BB}, \cite{Bo}, \cite{DM}, \cite{E}, \cite{IKZ},  \cite{Li}  and \cite{Q} have shown on some examples that the heat semigroup associated with certain subelliptic operators may satisfy functional inequalities that were only known to hold in elliptic situations. These examples have in common the property  that the subelliptic diffusion operator satisfies the generalized curvature dimension inequality that was introduced in \cite{BG1} in an abstract setting. As we will see in this work,  this curvature dimension inequality may also be used to prove the Poincar\'e inequality, the log-Sobolev inequality or the Gaussian logarithmic isoperimetric inequality for the invariant measure of a subelliptic diffusion operator in some interesting new situations.

\

Let us describe our framework and results in more details.
In this paper, $\bM$ will be a $C^\infty$ connected finite dimensional manifold endowed with a smooth measure $\mu$ and a second-order diffusion operator $L$ on $\M$, locally subelliptic in the sense of \cite{FP1} (see also \cite{JSC}), satisfying $L1=0$, 
\begin{equation*}
\int_\bM f L g d\mu=\int_\bM g Lf d\mu,\ \ \ \ \ \ \int_\bM f L f d\mu \le 0,
\end{equation*}
for every $f , g \in C^ \infty_0(\bM)$. 
We indicate with $\Gamma(f):=\Gamma(f,f)$ the \textit{carr\'e du champ}, that is the quadratic differential form defined by 
\begin{equation}\label{gamma}
\Gamma(f,g) =\frac{1}{2}(L(fg)-fLg-gLf), \quad f,g \in C^\infty(\bM).
\end{equation}

An absolutely continuous curve $\gamma: [0,T] \rightarrow \bM$ is said to be subunit for the operator $L$ if for every smooth function $f : \bM \to \mathbb{R}$ we have $ \left| \frac{d}{dt} f ( \gamma(t) ) \right| \le \sqrt{ (\Gamma f) (\gamma(t)) }$.  We then define the subunit length of $\gamma$ as $\ell_s(\gamma) = T$. Given $x, y\in \M$, we indicate with 
\[
S(x,y) =\{\gamma:[0,T]\to \M\mid \gamma\ \text{is subunit for}\ L, \gamma(0) = x,\ \gamma(T) = y\}.
\]
In this paper we assume that 
\[
S(x,y) \not= \varnothing,\ \ \ \ \text{for every}\ x, y\in \M.
\]
Under such assumption  it is easy to verify that
\begin{equation}\label{ds}
d(x,y) = \inf\{\ell_s(\gamma)\mid \gamma\in S(x,y)\},
\end{equation}
defines a true distance on $\M$. Furthermore, it is known that
\begin{equation}\label{di}
d(x,y)=\sup \left\{ |f(x) -f(y) | \mid f \in  C^\infty(\bM) , \| \Gamma(f) \|_\infty \le 1 \right\},\ \ \  \ x,y \in \bM.
\end{equation}
Throughout this paper we assume that the metric space $(\M,d)$ is complete.

\

In addition to the differential form \eqref{gamma}, we assume that $\M$ is endowed with another smooth symmetric bilinear differential form, indicated with $\Gamma^Z$, satisfying for $f,g \in C^\infty(\M)$
\[
\Gamma^Z(fg,h) = f\Gamma^Z(g,h) + g \Gamma^Z(f,h),
\] 
and $\Gamma^Z(f) = \Gamma^Z(f,f) \ge 0$. 

We make the following assumptions that will be in force throughout the paper:

\

\begin{itemize}
\item[(H.1)] There exists an increasing
sequence $h_k\in C^\infty_0(\bM)$   such that $h_k\nearrow 1$ on
$\bM$, and \[
||\Gamma (h_k)||_{\infty} +||\Gamma^Z (h_k)||_{\infty}  \to 0,\ \ \text{as} \ k\to \infty.
\]
\item[(H.2)]  
For any $f \in C^\infty(\bM)$ one has
\[
\Gamma(f, \Gamma^Z(f))=\Gamma^Z( f, \Gamma(f)).
\] 
 \end{itemize}
 \ 

As it has been proved in \cite{BG1}, the assumption (H.1) which is of technical nature, implies in particular that $L$ is essentially self-adjoint on $C^\infty_0(\bM)$. The assumption (H.2) is more subtle and is crucial for the validity of most  the subsequent results: It is discussed in details in \cite{BG1} in several geometric examples. 
Let us consider
\begin{equation}\label{gamma2}
\Gamma_{2}(f,g) = \frac{1}{2}\big[L\Gamma(f,g) - \Gamma(f,
Lg)-\Gamma (g,Lf)\big],
\end{equation}
\begin{equation}\label{gamma2Z}
\Gamma^Z_{2}(f,g) = \frac{1}{2}\big[L\Gamma^Z (f,g) - \Gamma^Z(f,
Lg)-\Gamma^Z (g,Lf)\big].
\end{equation}
As for $\Gamma$ and $\Gamma^Z$, we will freely use the notations  $\Gamma_2(f) = \Gamma_2(f,f)$, $\Gamma_2^Z(f) = \Gamma^Z_2(f,f)$.

\begin{definition}
We say that $L$ satisfies the  \emph{generalized curvature-dimension inequality} \emph{CD}$(\rho_1,\rho_2,\kappa,d)$ if 
there exist constants $\rho_1 \in \mathbb{R} $,  $\rho_2 >0$, $\kappa \ge 0$, and $0< d \le \infty$ such that the inequality 
\begin{equation*}
\Gamma_2(f) +\nu \Gamma_2^Z(f) \ge \frac{1}{d} (Lf)^2 +\left( \rho_1 -\frac{\kappa}{\nu} \right) \Gamma(f) +\rho_2 \Gamma^Z(f)
\end{equation*}
 holds for every  $f\in C^\infty(\bM)$ and every $\nu>0$, where $\Gamma_2$ and $\Gamma_2^Z$ are defined by (\ref{gamma2}) and (\ref{gamma2Z}).
\end{definition}

\begin{remark}
Of course, it is understood in the previous definition that  \emph{CD}$(\rho_1,\rho_2,\kappa,\infty)$,  means 
\[
\Gamma_2(f) +\nu \Gamma_2^Z(f) \ge \left( \rho_1 -\frac{\kappa}{\nu} \right) \Gamma(f) +\rho_2 \Gamma^Z(f)
\]
\end{remark}

\

The purpose of our work is to understand the functional inequalities that are satisfied by the invariant measure $\mu$ under the assumption that the generalized curvature dimension inequality is satisfied.  Let us observe that unlike \cite{BG1} and \cite{BBG}, where the authors focussed on functional inequalities involving in a crucial way the dimension $d$,  here we shall mainly be interested in functional inequalities that are independent from the dimension $d$.

\

The paper is organized as follows. The purpose of Section 2 is to prove the following theorem:

\begin{theorem}\label{poinc_intro}
Assume that $L$ satisfies the  \emph{generalized curvature-dimension inequality} \emph{CD}$(\rho_1,\rho_2,\kappa,\infty)$ with  $\rho_1 > 0$,  $\rho_2 >0$ and $\kappa \ge 0$.
\begin{itemize}
\item The measure $\mu$ is finite and the following Poincar\'e inequality holds:
\[
\int_\bM f^2 d\mu -\left( \int_\bM f d\mu\right)^2 \le \frac{\kappa+\rho_2}{\rho_1 \rho_2} \int_\bM \Gamma(f) d\mu, \quad f \in \mathcal{D}(L).
\]
\item If $\mu$ is a probability measure, that is $\mu(\bM)=1$, then for $f \in C_0(\bM)$,
 \[
\int_\bM f^2 \ln f^2 d\mu -\int_\bM f^2 d\mu \ln \int_\bM f^2 d\mu \le\frac{2(\kappa+\rho_2)}{\rho_1 \rho_2 } \left( \int_\bM \Gamma(f) d\mu+\frac{\kappa+\rho_2}{\rho_1}  \int_\bM \Gamma^Z(f) d\mu\right).
\]
\end{itemize}
\end{theorem}

In Section 3, we will prove the following theorem which is a subelliptic analogue of a famous result due to F. Y. Wang \cite{W1}.
\begin{theorem}\label{wang intro}
Assume that the measure $\mu$ is a probability measure and that $L$ satisfies the \emph{generalized curvature dimension inequality} \emph{CD}$(\rho_1,\rho_2,\kappa,\infty)$ for some $\rho_1 \in \mathbb{R} $,  $\rho_2 >0$, $\kappa \ge 0$.  Assume moreover that
\[
\int_\bM e^{\lambda d^2 (x_0,x )} d\mu(x) <+\infty,
\]
for some $x_0 \in \bM$ and $\lambda >  \frac{\rho_1^-}{2}$, then there is a constant $\rho_0>0$ such that for every function $f  \in C^ \infty_0(\bM)$,
\[
\int_\bM f^2 \ln f^2 d\mu -\int_\bM f^2 d\mu \ln  \int_\bM f^2 d\mu \le \frac{2}{\rho_0}  \int_\bM \Gamma(f) d\mu.
\]
\end{theorem}

In Section 4, adapting some methods of Bobkov-Gentil-Ledoux \cite{BGL}, we prove an analogue of an Otto-Villani theorem \cite{OV}.  We recall that $L^2$-Wasserstein distance of two measures $\nu_1$ and $\nu_2$ on $\bM$ is defined by
 $$
 \W_2(\nu_1,\nu_2) ^2 = \inf_\Pi \int_\bM d^2(x,y) d\Pi (x,y)
$$
where the infimum is taken over all coupling of $\nu_1$ and $\nu_2$ that is on all probability measures $\Pi$ on $\bM\times \bM$ whose marginals are respectively $\nu_1$ and  $\nu_2$.

\begin{theorem}\label{HWI modified}
Assume that the measure $\mu$ is a probability measure and  that $L$ satisfies the generalized curvature dimension inequality \emph{CD}$(\rho_1,\rho_2,\kappa,\infty)$ for some $\rho_1 \in \mathbb{R} $,  $\rho_2 >0$, $\kappa \ge 0$. If the quadratic transportation cost inequality 
\begin{equation}
\W_2 (\mu, \nu)^2 \leq c \Ent_\mu \left( \frac{d\nu}{d\mu}\right) 
\end{equation} 
 is satisfied for every absolutely continuous probability measure $\nu$  with a constant $c<\frac{2}{\rho_1^-}$, then the following modified log-Sobolev inequality 
$$\Ent_\mu(f) \leq C_1 \int_\bM \frac{\Ga(f)}{f} d\mu +C_2\int_\bM \frac{\Ga^Z(f)}{f} d\mu $$
holds for some constants  $C_1$ and $C_2$ depending only on $c ,\rho_1, \kappa,\rho_2$.
\end{theorem}

Finally, the goal of the Section 5 is to study isoperimetric inequalities. We will prove the following result which is a subelliptic generalization of a theorem due to Ledoux \cite{L2}:

\begin{theorem} 

Assume that the measure $\mu$ is a probability measure,  that $L$ satisfies the generalized curvature dimension inequality \emph{CD}$(\rho_1,\rho_2,\kappa,\infty)$ for some $\rho_1 \in \mathbb{R} $,  $\rho_2 >0$, $\kappa \ge 0$ and that $\mu $ satisfies the log-Sobolev inequality:
\begin{equation}
\int_\bM f^2 \ln f^2 d\mu -\int_\bM f^2 d\mu \ln  \int_\bM f^2 d\mu \le \frac{2}{\rho_0}  \int_\bM \Gamma(f) d\mu, \quad f \in C^\infty_0(\bM)
\end{equation} for all smooth functions $f \in C_0^\infty(\bM)$.
Let $A$ be a set of the  manifold $\M$  which has a finite perimeter $P(A)$ and such that $0\leq \mu(A)\leq \frac{1}{2}$, then 
$$
P( A)\geq \frac{\ln 2 }{ 4\left(3+\frac{2\kappa}{\rho_2}\right)  } \min \left(\sqrt \rho_0 ,\frac{\rho_0} {\sqrt {\rho_1^-}}\right)   \mu(A)\left(\ln \frac{1}{\mu(A)} \right)^\frac{1}{2}.
$$
\end{theorem}

To conclude this introduction, let us now turn to the fundamental question of examples to which our above results apply. We refer the reader to \cite{BG1} for more details about most of the examples we discuss below.

A first observation is that if $\bM$ is a $n$-dimensional complete Riemannian manifold and $L$ is the  Laplace-Beltrami operator ,  the assumptions (H.1) and (H.2) hold trivially  with $\Gamma^Z = 0$ . Indeed,  the  assumption (H.1) is  satisfied as a consequence of the completeness of $\bM$ and the  assumption (H.2) is trivially  satisfied. In this example,   the generalized curvature dimension inequality \emph{CD}$(\rho_1, 1,0,n)$ is implied by (and it is in fact equivalent to) the assumption that the Ricci curvature of $\bM$ satisfies the lower bound Ric $\ \ge \rho_1$.

Besides Laplace-Beltrami operators on complete Riemannian manifolds with Ricci curvature bounded from below, a wide class of examples is given by  sub-Laplacians on  Sasakian manifolds. 
Let $\mathbb{M}$ be a complete strictly pseudo convex  CR Sasakian manifold with real dimension $2n +1$. Let $\theta$ be a pseudo-hermitian form on $\mathbb{M}$ with respect to which the Levi form is positive definite. The kernel of $\theta$ determines an horizontal bundle $\mathcal H$. Denote now  by $T$ the Reeb vector field on $\bM$, i.e., the characteristic direction of $\theta$. We denote by $\nabla$ the Tanaka-Webster connection of $\mathbb{M}$.
We recall that the CR manifold $(\bM,\theta)$ is called Sasakian if the pseudo-hermitian torsion of $\nabla$ vanishes, in the sense that $\mathbf{T}(T,X) = 0$, for every $X\in \mathcal H$.
For instance the standard CR structures on the  Heisenberg group $\mathbb{H}_{2n+1}$ and the sphere $\mathbb{S}^{2n+1}$ are Sasakian. On CR manifolds, there is a canonical subelliptic diffusion operator which is called the CR sub-Laplacian. It plays the same role in CR geometry as the Laplace-Beltrami operator does in Riemannian geometry.
In this framework we have the following result that shows the relevance of the generalized curvature dimension inequality.
\begin{proposition}\cite{BG1}
Let $(\bM,\theta)$ be a \emph{CR} manifold  with real dimension $2n+1$ and vanishing Tanaka-Webster torsion, i.e., a Sasakian manifold. If 
for every $x\in \bM$ the Tanaka-Webster Ricci tensor satisfies the bound  
\[
\emph{Ric}_x(v,v)\ \ge \rho_1|v|^2,
\]
for every horizontal vector $v\in \mathcal H_x$, then, for the CR sub-Laplacian of $\bM$,  the curvature-dimension inequality \emph{CD}$(\rho_1,\frac{d}{4},1,d)$ holds with $d = 2n$ and $\Gamma^Z(f)=(Tf)^2$.
\end{proposition}

In addition to sub-Laplacians on  Heisenberg groups, more generally, the sub-Laplacian on any Carnot group of step 2 has been shown to satisfy the generalized curvature-dimension inequality \emph{CD}$(0,\rho_2,\kappa,d)$, for some values of the parameters $\rho_2$ and $\kappa$. Let us mention that recently, in \cite{BGM}, the authors study sub-Laplacians in infinite dimensional Heisenberg type groups and show that a generalized curvature dimension inequality is satisfied with $d=+\infty$. In that case the assumption  (H.1) is of course not satisfied but is somehow replaced by the existence of nice and uniform finite-dimensional approximations, so that with suitable modifications the results of the present paper may be used.  For infinite dimensional situations, we  also point out the reader to the work   \cite{IP}  that uses completely different methods .

Another interesting example, which has recently been highlighted in different contexts by several works  (see \cite{W5,W6,W7,GW}, see also \cite{GS})  is given by the Grushin operator on $\R^{2n}$. It is defined by $$L =\sum_{i=1}^n\left(  \frac{\partial^2}{ \partial x_i^2 }+ \frac{ \|x\|^2}{2}   \frac{\partial^2}{ \partial y_i^2 }\right)$$
where $\|x\|^2= x_1^2+\dots +x_n^2$ for $x=(x_1,\dots, x_n) \in \R^n$ .This operator admits the Lebesgue measure  $\lambda$ as invariant and symmetric measure.
If we set $X_i=  \frac{\partial}{ \partial x_i }$, $Y_{i,j}= x_j \frac{\partial}{ \partial y_i }$ and $Z_i=\frac{\partial}{ \partial y_i }$, we can write  this operator as 
$$
L= \sum_{i=1}^n X_i^2 + \sum_{i,j=1}^n Y_{i,j}^2 = - \sum_{i=1}^n  X_i^* X_i -\sum_{i,j=1}^n Y_{i,,j}^* Y_{i,j}.
$$
The  only non zero Lie bracket relations are 
$$
[X_i,Y_{i,j}]=Z_j \textrm{ for }1\leq i,j \leq n. 
$$
This algebra structure is then exactly the one of a Carnot group of step 2 and  the criterion  \emph{CD}$(0,\rho_2,\kappa,n+n^2)$ therefore holds with $\Ga^Z(f,f)=\sum_{i=1}^n \left(\frac{\partial f}{ \partial y_i } \right)^2$ and some constant $\rho_2$. Also it is easy to see that assumption (H.1) and (H.2) are satisfied in that case.
Let us however observe that more general Grushin operators are considered in \cite{W7},  and that they can not be handled at the moment with our methods, since their Lie algebra correspond to a Carnot group of step  higher than 2.  
Finally, we mention that some close results are obtained in \cite{GW}  for  Fokker-Planck type operators. In those examples, that typically do not satisfy the generalized curvature dimension inequality studied in this work, the hypoellipticity of the operator stems from its first order part; a situation radically different from the examples discussed above.

\

\textbf{Acknowledgements:} The authors would like to thank an anonymous referee for his careful reading and his interesting comments that helped improve the present paper.

\section{Spectral gap and modified log-Sobolev inequalities}

 Throughout this section, we assume that the operator $L$ satisfies the curvature dimension inequality \emph{CD}$(\rho_1,\rho_2,\kappa,\infty)$ for some $\rho_1 >0 $,  $\rho_2 >0$, $\kappa \ge 0$.
 
The main tool to prove the  fore mentioned theorems, is the heat semigroup $P_t=e^{tL}$, which is defined using the spectral theorem. Since $L$ satisfies the curvature dimension inequality  \emph{CD}$(\rho_1,\rho_2,\kappa,\infty)$, this semigroup is stochastically complete (see \cite{BG1}), i.e. $P_t 1 =1$. Moreover, thanks to the hypoellipticity of $L$, for $f \in L^p(\bM)$,  $1 \le p \le \infty$, the function $(t,x) \rightarrow P_t f(x)$ is
smooth on $\mathbb{M}\times (0,\infty) $ and
\[ P_t f(x)  = \int_{\mathbb M} p(x,y,t) f(y) d\mu(y)\] where $p(x,y,t) = p(y,x,t) > 0$ is the so-called heat
kernel associated to $P_t$.

Henceforth, in all the paper, we denote
\[
C^\infty_b(\bM) = C^\infty(\bM) \cap L^\infty(\bM).
\]
For $\varepsilon>0$ we denote by $\mathcal{A}_\varepsilon$ the set of functions $f \in C^\infty_b(\M)$ such that
\[
f=g+\varepsilon,
\]
for some $\varepsilon >0$ and some $g \in C^\infty_b(\M)$, $g \ge 0$, such that $g, \sqrt{\Gamma(g)}, \sqrt{\Gamma^Z(g)} \in L^2(\bM)$. As shown in \cite{BG1}, this set is stable under the action of $P_t$, i.e.,  if  $f\in \mathcal{A}_\varepsilon$, then $P_t f \in  \mathcal{A}_\varepsilon$.

Our goal is to prove Theorem \ref{poinc_intro}. In that direction, we first establish a useful gradient bound for $P_t$.

\begin{proposition}\label{log-bound}
Let $\varepsilon >0$ and $f \in \mathcal A_\varepsilon$. 
 For $x\in \M$,  $t \ge 0$ one has:
 \[
( P_t f) \Gamma( \ln P_t f) +\frac{\kappa+\rho_2}{\rho_1} (P_tf)  \Gamma^Z(\ln P_t f) \le e^{-2\frac{\rho_1 \rho_2 t}{\kappa+\rho_2}} \left( P_t ( f \Gamma(\ln f)) +\frac{\kappa+\rho_2}{\rho_1}  P_t ( f \Gamma^Z(\ln f))\right)
 \]

\end{proposition}
\begin{proof}
Let us fix $T>0$ once time for all in the following proof. Given a function $f\in \mathcal A_\varepsilon$, for $0\le t\le T$ we introduce the entropy functionals

\[
\phi_1 (x,t)=(P_{T-t} f) (x)\Gamma (\ln P_{T-t}f)(x),
\]
\[
\phi_2 (x,t)= (P_{T-t} f)(x) \Gamma^Z (\ln P_{T-t}f)(x),
\]
which are defined on $\M\times [0,T]$.   As it has been proved in \cite{BG1}, a direct computation shows  that
\[
L\phi_1+\frac{\partial \phi_1}{\partial t} =2 (P_{T-t} f) \Gamma_2 (\ln P_{T-t}f). 
\]
and
\[
L\phi_2+\frac{\partial \phi_2}{\partial t} =2 (P_{T-t} f) \Gamma_2^Z (\ln P_{T-t}f).
\]
Let us observe that for the second equality the Hypothesis (H.2) is used in a crucial way.

Consider now the function
\begin{align*}
\phi (x,t)&= a(t) \phi_1 (x,t)+b(t) \phi_2(x,t) \\
 & =a(t)(P_{T-t} f) (x)\Gamma (\ln P_{T-t}f)(x)+b(t)(P_{T-t} f) (x) \Gamma^Z (\ln P_{T-t}f)(x),
\end{align*}
where $a$ and $b$ are two non negative functions  that will be chosen later.
Applying the generalized curvature-dimension inequality \emph{CD}$(\rho_1,\rho_2,\kappa,\infty)$, we obtain
\begin{align*}
  L\phi+\frac{\partial \phi}{\partial t} =&
a' (P_{T-t} f) \Gamma (\ln P_{T-t}f)+b' (P_{T-t} f) \Gamma^Z (\ln P_{T-t}f)  \\
 &+2a (P_{T-t} f) \Gamma_2 (\ln P_{T-t}f)+2b (P_{T-t} f) \Gamma_2^Z (\ln P_{T-t}f) \\
&\ge  \left(a'+2\rho_1 a -2\kappa \frac{a^2}{b}\right)(P_{T-t} f) \Gamma (\ln P_{T-t}f)  +(b'+2\rho_2 a) (P_{T-t} f)  \Gamma^Z (\ln P_{T-t}f).
\end{align*}
Let us now chose
\[
b(t)=e^{-\frac{2\rho_1 \rho_2 t}{\kappa+\rho_2}}
\]
and
\[
a(t)=-\frac{b'(t)}{2\rho_2},
\]
so that
\begin{align*}
b'+2\rho_2 a=0
\end{align*}
and
\begin{align*}
a'+2\rho_1 a -2\kappa \frac{a^2}{b}  =0.
\end{align*}
With this choice, we get
\[
 L\phi+\frac{\partial \phi}{\partial t}  \ge 0.
\]
and therefore from a comparison theorem for parabolic partial differential equations (see for instance p.52 in \cite{Fried} or Proposition 3.2 in \cite{BG1})  we have
\[
P_T(\phi(\cdot,T))(x) \ge \phi(x,0) .
\]
Since,
\[
\phi(x,0)=a(0)(P_{T} f) (x)\Gamma (\ln P_{T}f)(x)+b(0)(P_{T} f) (x) \Gamma^Z (\ln P_{T}f)(x)
\]
and
\[
P_T(\phi(\cdot,T))(x) =a(T)P_{T}( f \Gamma (\ln f))(x)+b(T)P_{T}( f \Gamma^Z (\ln f))(x),
\]
the proof is completed.
\end{proof}

A similar proof as above also provides the following:

\begin{proposition}\label{grad-bound}
 Let $f \in L^2(\bM)$ such that $f \in C^\infty(\bM)$ and $\Gamma(f), \Gamma^Z(f)\in L^1(\bM)$. For $x\in \M$,  $t \ge 0$ one has:
 \[
  \Gamma( P_t f) +\frac{\kappa+\rho_2}{\rho_1} \Gamma^Z( P_t f) \le e^{-2\frac{\rho_1 \rho_2 t}{\kappa+\rho_2}} \left( P_t (  \Gamma( f)) +\frac{\kappa+\rho_2}{\rho_1}  P_t ( \Gamma^Z( f))\right)
 \]
\end{proposition}

\begin{proof}
We introduce
\[
\phi_1 (x,t)=\Gamma (P_{T-t}f)(x),
\]
\[
\phi_2 (x,t)= \Gamma^Z ( P_{T-t}f)(x),
\]
and observe that
\[
L\phi_1+\frac{\partial \phi_1}{\partial t} =2  \Gamma_2 ( P_{T-t}f). 
\]
and
\[
L\phi_2+\frac{\partial \phi_2}{\partial t} =2  \Gamma_2^Z ( P_{T-t}f).
\]
The conclusion is then reached by following the lines of the proof of Proposition \ref{log-bound}.
\end{proof}

A first interesting consequence of the above functional inequalities is the fact that $\rho_1>0$ implies that the invariant measure is finite.
 
\begin{corollary}
The measure $\mu$ is finite, i.e.
$\mu(\bM) <+\infty$
and for every $x \in \bM$, $f \in L^2(\M)$,

\[
P_t f (x) \to_{t \to +\infty} \frac{1}{\mu(\bM)} \int_\bM f d\mu.
\]
\end{corollary}
\begin{proof}
 Let $f,g \in C^\infty_0(\bM)$, we have
\begin{align*}
\int_{\bM} (P_t f -f) g d\mu = \int_0^t \int_{\bM}\left(
\frac{\partial}{\partial s} P_s f \right) g d\mu ds= \int_0^t
\int_{\bM}\left(L P_s f \right) g d\mu ds=- \int_0^t \int_{\bM}
\Gamma ( P_s f , g) d\mu ds.
\end{align*}
By means of Proposition \ref{grad-bound}, and Cauchy-Schwarz inequality, we
find
\begin{equation}\label{P1}
\left| \int_{\bM} (P_t f -f) g d\mu \right| \le \left(\int_0^t
e^{-\frac{\rho_1 \rho_2 s}{\kappa+\rho_2}} ds\right) \sqrt{ \| \Gamma (f) \|_\infty +\frac{\kappa+\rho_2}{\rho_1} \|
\Gamma^Z (f) \|_\infty } \int_{\bM}\Gamma (g)^{\frac{1}{2}}d\mu.
\end{equation}
Now it is seen from spectral theorem that in $L^2(\bM)$ we have  a convergence $P_t f \to P_\infty f$, where $P_\infty f$ belongs to the domain of $L$. Moreover $LP_\infty f=0$. By hypoellipticity of $L$ we deduce that $P_\infty f$ is a smooth function. Since $LP_\infty f=0$, we have $\Gamma(P_\infty f)=0$ and therefore $P_\infty f $ is constant.

 Let us now assume that $\mu(\bM)=+\infty$. This implies in particular that $P_\infty f =0$  because no constant besides $0$ is in $L^2(\bM)$. Using then (\ref{P1}) and letting $t \to +\infty$, we infer
 \begin{equation*}
\left| \int_{\bM} f g d\mu \right| \le \left(\int_0^{+\infty}
e^{-2\frac{\rho_1 \rho_2 s}{\kappa+\rho_2}} ds\right) \sqrt{ \| \Gamma (f) \|_\infty +\frac{\kappa+\rho_2}{\rho_1} \|
\Gamma^Z (f) \|_\infty } \int_{\bM}\Gamma (g)^{\frac{1}{2}}d\mu.
\end{equation*}
Let us assume $g \ge 0$, $g \ne 0$ and take for $f$ the sequence $h_n$ from Assumption (H.1). Letting $n \to \infty$, we deduce
\[
\int_\bM g d\mu \le 0,
\]
which is clearly absurd. As a consequence $\mu (\bM) <+\infty$. 

The invariance of $\mu$ implies then 
\[
\int_\bM P_\infty f d\mu =\int_\bM f d\mu,
\]
and thus
\[
P_\infty f =\frac{1}{\mu(\bM)} \int_\bM f d\mu.
\]
Finally, using the Cauchy-Schwarz inequality, we find that for $x \in \bM$, $f \in L^2(\bM)$, $s,t,\tau \ge 0$,
\begin{align*}
| P_{t+\tau} f (x)-P_{s+\tau} f (x) | & = | P_\tau (P_t f -P_s f) (x) | \\
 &=\left| \int_\bM p(\tau, x, y) (P_t f -P_s f) (y) \mu(dy) \right| \\
 &\le \int_\bM p(\tau, x, y)^2 \mu(dy) \| P_t f -P_s f\|^2_2 \\
 &\le p(2\tau,x,x) \| P_t f -P_s f\|^2_2.
\end{align*}
Thus, we also have 
\[
P_t f (x) \to_{t \to +\infty} \frac{1}{\mu(\bM)} \int_\bM f d\mu.
\]
\end{proof}

We also deduce a spectral gap inequality:
\begin{corollary}
For every $f $ in the domain of $L$,
\[
\int_\bM f^2 d\mu -\left( \int_\bM f d\mu\right)^2 \le \frac{\kappa+\rho_2}{\rho_1 \rho_2} \int_\bM \Gamma(f) d\mu.
\]
\end{corollary}

\begin{proof}

Let $f\in C^\infty(\bM)$ with a compact support. By Proposition \ref{grad-bound}, we have for $t\geq 0$
$$
 \int_{\bM} \Ga(P_t f,P_t f) d\mu \leq C(f) e^{-\frac{2\rho_1 \rho_2 t}{\kappa+\rho_2}}.
$$
with 
$$
C(f)=\int_\bM \Gamma(f,f)+ \frac{\kappa+\rho_2}{\rho_1} \Gamma^Z(f,f) d\mu.
$$
By the spectral theorem, one has
$$
 \int_{\bM} \Ga(P_t f,P_t f) d\mu = \int_0^\infty \lambda e^{-2 \lambda t} dE_\lambda(f)
$$
and 
$$  \int_{\bM} \Ga( f, f) d\mu = \int_0^\infty \lambda dE_\lambda(f)
$$
where $dE_\lambda$ is the spectral measure associated to $-L$. 
Thus, by Holder inequality, for $0 \leq s \leq t$
\beas
\int_{\bM} \Ga(P_s,P_s f) d\mu = \int_0^\infty \lambda e^{-2 \lambda s} dE_\lambda(f) &\leq & \left(\int_0^\infty \lambda e^{-2 \lambda s} d E_\lambda(f)\right)^\frac{s}{t} \left(\int_0^\infty \lambda  dE_\lambda(f)\right) ^\frac{t-s}{t} \\ 
 &\leq & C(f)^\frac{s}{t}  e^{-\frac{2\rho_1 \rho_2 s} {\kappa+\rho_2}} \left(\int_{\bM} \Ga( f, f) d\mu\right)^ \frac{t-s}{t} .
\eeas
Letting $t\to \infty$ gives 
$$
\int_{\bM} \Ga(P_s,P_s f) d\mu   \leq  e^{-\frac{2\rho_1 \rho_2 s} {\kappa+\rho_2}} \int_{\bM} \Ga( f, f) d\mu
$$
for all $ C^\infty $ function with a compact support. Since this space is dense in the domain of the Dirichlet form, it implies the desired Poincar\'e inequality.
\end{proof}

Finally, we also deduce a modified log-Sobolev inequality that  involves a vertical term:

\begin{corollary}
Let us assume $\mu(\bM)=1$. For $f \in C_0(\bM)$,
\[
\int_\bM f^2 \ln f^2 d\mu -\int_\bM f^2 d\mu \ln \int_\bM f^2 d\mu \le\frac{2(\kappa+\rho_2)}{\rho_1 \rho_2 } \left( \int_\bM \Gamma(f) d\mu+\frac{\kappa+\rho_2}{\rho_1}  \int_\bM \Gamma^Z(f) d\mu\right).
\]
\end{corollary}

\begin{proof}
Let $g \in \mathcal{A}_\varepsilon$. We have
\begin{align*}
\int_\bM g \ln g d\mu -\int_\bM g d\mu \ln \int_\bM g d\mu &=-\int_0^{+\infty} \frac{\partial}{\partial t} \int_\bM P_t g \ln P_t g d\mu dt \\
 & =-\int_0^{+\infty}   \int_\bM L P_t g \ln P_t g d\mu dt \\
 &=\int_0^{+\infty}   \int_\bM \frac{\Gamma(P_t g)}{P_t g} d\mu dt \\
 &=\int_0^{+\infty}   \int_\bM P_t g \Gamma( \ln P_t g) d\mu dt \\
 &\le \int_0^{+\infty} e^{-2\frac{\rho_1 \rho_2 t}{\kappa+\rho_2}} dt \int_\bM  \left(  g \Gamma(\ln g) +\frac{\kappa+\rho_2}{\rho_1}   g \Gamma^Z(\ln g))\right) d\mu \\
 &\le \frac{\kappa+\rho_2}{2\rho_1 \rho_2 } \int_\bM  \left(  \frac{ \Gamma( g)}{g} +\frac{\kappa+\rho_2}{\rho_1}   \frac{ \Gamma^Z( g)}{g} \right) d\mu 
\end{align*}
Let now $f \in C_0(\bM)$ and consider $g=\varepsilon +f^2 \in \mathcal{A}_\varepsilon$. Using the previous inequality and letting $\varepsilon \to 0$, yields
\[
\int_\bM f^2 \ln f^2 d\mu -\int_\bM f^2 d\mu \ln \int_\bM f^2 d\mu \le\frac{2(\kappa+\rho_2)}{\rho_1 \rho_2 } \left( \int_\bM \Gamma(f) d\mu+\frac{\kappa+\rho_2}{\rho_1}  \int_\bM \Gamma^Z(f) d\mu\right).
\]
\end{proof}

\section{Wang inequality for the heat semigroup and log-Sobolev inequality}

Throughout this section, we assume that the operator $L$ satisfies the curvature dimension inequality \emph{CD}$(\rho_1,\rho_2,\kappa,\infty)$ for some $\rho_1 \in \mathbb{R} $,  $\rho_2 >0$, $\kappa \ge 0$. We shall denote $\rho_1^-=\max ( -\rho_1, 0)$.

\

Our main goal is to prove Theorem \ref{wang intro}.

\subsection{Reverse log-Sobolev inequalities}

\begin{proposition}\label{reverse_logsob}
Let $\varepsilon >0$ and $f \in \mathcal A_\varepsilon$. 
 For $x\in \M$,  $t>0$ one has
\[
t P_t f(x) \Gamma(\ln P_t f)(x)  +\rho_2 t^2  P_t f(x) \Gamma^Z(\ln P_t f)(x) \le \left(1+\frac{2\kappa}{\rho_2}+ 2 \rho^-_1 t \right) \big[P_t ( f\ln f )(x) -P_tf(x)\ln P_t f(x)\big].
\]
\end{proposition}

\begin{proof}
We may assume $\rho_1 \le 0$. 
We proceed similarly to the proof of Proposition \ref{log-bound}. 
Let  $f\in \mathcal A_\varepsilon$,  $0\le t\le T$  and 
\[
\phi_1 (x,t)=(P_{T-t} f) (x)\Gamma (\ln P_{T-t}f)(x),
\]
\[
\phi_2 (x,t)= (P_{T-t} f)(x) \Gamma^Z (\ln P_{T-t}f)(x),
\]
As before, we  consider the function
\begin{align*}
\phi (x,t)&= a(t) \phi_1 (x,t)+b(t) \phi_2(x,t) \\
 & =a(t)(P_{T-t} f) (x)\Gamma (\ln P_{T-t}f)(x)+b(t)(P_{T-t} f) (x) \Gamma^Z (\ln P_{T-t}f)(x),
\end{align*}
where $a$ and $b$ are to be later chosen.
As already seen, applying the generalized curvature-dimension inequality \emph{CD}$(\rho_1,\rho_2,\kappa,\infty)$, we obtain
\begin{align*}
  L\phi+\frac{\partial \phi}{\partial t} \ge  \left(a'+2\rho_1 a -2\kappa \frac{a^2}{b}\right)(P_{T-t} f) \Gamma (\ln P_{T-t}f)  +(b'+2\rho_2 a) (P_{T-t} f)  \Gamma^Z (\ln P_{T-t}f).
\end{align*}

The idea is now to chose the functions $a$ and $b$ in such a way that
\[
b'+2\rho_2 a=0
\]
and
\[
a'+2\rho_1 a -2\kappa \frac{a^2}{b}  \ge C
\]
where $C$ is a constant independent from $t$. This leads to the  candidates
\[
a(t)=\frac{1}{\rho_2} (T -t)
\]
and
\[
b(t)=(T -t)^2,
\]
for which we obtain
\[
C=-\frac{1}{\rho_2}-\frac{2 \kappa}{\rho_2^2} +\frac{2\rho_1}{\rho_2} T.
\]
For this  choice of $a$ and $b$, we obtain
\[
 L\phi+\frac{\partial \phi}{\partial t} \ge C (P_{T-t} f) \Gamma (\ln P_{T-t}f) 
\]
The comparison principle for parabolic partial differential equations leads then to
\[
P_T(\phi(\cdot,T))(x) \ge \phi (0,x) +C \int_0^T P_t ((P_{T-t} f) \Gamma (\ln P_{T-t}f)) (x) dt.
\]
It is now seen that
\[
\int_0^T  P_t ((P_{T-t} f) \Gamma (\ln P_{T-t}f)) (x)  dt=P_T ( f\ln f )(x) -P_Tf(x) \ln P_T f(x).
\]
which yields
\begin{align*}
  & T P_T f(x) \Gamma(\ln P_T f)(x)  +\rho_2 T^2  P_T f(x) \Gamma^Z(\ln P_T f)(x)   \\
 \le & \left(1+\frac{2\kappa}{\rho_2}+ 2 \rho^-_1 T \right) \big[P_T ( f\ln f )(x) -P_Tf(x)\ln P_T f(x)\big].
\end{align*}
\end{proof}

Using a similar argument, we may prove the following:
\begin{proposition}\label{rev-poinc}
Let  $f \in C_0^\infty(\bM)$, then
for $x\in \M$,  $t>0$ one has
\[
t \Gamma(P_t f)(x)  +\rho_2 t^2   \Gamma^Z( P_t f)(x) \le \frac{1}{2} \left(1+\frac{2\kappa}{\rho_2}+ 2 \rho^-_1 t \right) \big[P_t ( f^2 )(x) -P_tf(x)^2 \big].
\]
\end{proposition}

As a consequence, we get the following useful regularization bound that will be later used:
\begin{corollary}\label{reg-bound}
Let  $f \in C_0^\infty(\bM)$, then for all $t>0$,
$$
\Vert\sqrt{\Ga(P_t f)} \Vert_\infty  \leq  \left( \frac{\frac{1}{2}+\frac{\kappa}{\rho_2}+  \rho^-_1 t}{t} \right)^\frac{1}{2} \Vert f\Vert_\infty.
$$
\end{corollary}

\subsection{Wang inequality}

An important by-product of the reverse log-Sobolev inequality that was proved in the previous section (Proposition \ref{reverse_logsob}) is the following inequality that was first observed by F.Y. Wang \cite{W1} in a Riemannian framework.

\begin{proposition}\label{Wang inequality}
Let $\alpha>1$. For $f \in L^\infty(\bM)$,   $f \ge 0$, $t>0$, $x,y \in \bM$,
$$
(P_tf)^\alpha (x) \leq P_t(f^\alpha) (y) \exp \left(  \frac{\alpha}{\alpha-1} \left(\frac{ 1+\frac{2\kappa}{\rho_2}+ 2\rho^-_1 t }{4t}\right) d^2(x,y)\right).
$$
\end{proposition}

\begin{proof}
 We first assume $f\in \mathcal{A}_\varepsilon$.

 Consider a subunit curve $\gamma:[0,T] \rightarrow \bM$ such that $\gamma(0)=x$, $\gamma(T)=y$.  Let $\alpha >1$ and $\beta(s)=1+(\alpha-1)\frac{s}{T}$, $0 \le s \le T$.  Let
\[
\phi (s)=\frac{\alpha}{\beta (s)}  \ln P_t f^{\beta (s)} (\gamma(s)), \quad 0 \le s \le T.
\]
where $t>0$ is fixed. Differentiating with respect to $s$ and using  then Proposition  \ref{reverse_logsob} yields
\begin{align*}
\phi' (s) & \ge \frac{\alpha (\alpha-1)}{T \beta(s)^2} \frac{P_t (f^{\beta(s)} \ln f^{\beta(s)}) -(P_t f^{\beta(s)}) \ln P_t f^{\beta(s)}  }{P_t f^{\beta(s)} } -\frac{\alpha}{\beta (s)}  \sqrt{ \Gamma( \ln P_t f^{\beta(s)})}  \\
 & \ge \frac{\alpha (\alpha-1) t }{T\beta(s)^2 \left( 1+\frac{2\kappa}{\rho_2} +2\rho_1^- t \right)} \Gamma(\ln  P_t f^{\beta(s)}) -\frac{\alpha}{\beta (s)}  \sqrt{ \Gamma( \ln P_t f^{\beta(s)})}.
\end{align*}
Now, for every $\lambda >0$,
\[
  - \sqrt{ \Gamma( \ln P_t f^{\beta(s)})} \ge -\frac{1}{2\lambda^2} \Gamma(\ln  P_t f^{\beta(s)})-\frac{\lambda^2}{2} .
\]
If we chose
\[
\lambda^2=\frac{\left( 1+\frac{2\kappa}{\rho_2} +2\rho_1^- t\right) }{2(\alpha -1) t } T \beta (s)
\]
we infer
\[
\phi' (s) \ge - \frac{\alpha \left(1+\frac{2\kappa}{\rho_2} +2\rho_1^- t\right) }{4(\alpha-1)t} T .
\]
Integrating from $0$ to $L$ yields
\[
\ln P_t (f^\alpha) (y) -\ln (P_t f)^\alpha (x) \ge - \frac{\alpha \left(1+\frac{2 \kappa}{\rho_2}+2\rho_1^- t \right) }{4(\alpha-1)t} T^2 .
\]
Minimizing then  $T^2$ over the set of subunit curves such that $\gamma(0)=x$ and $\gamma(T)=y$ gives the claimed result.

If $f \in L^\infty(\bM)$, $f \ge 0$,  then for $\varepsilon >0$, $n \ge 0$, and $\tau>0$, the function $\varepsilon +h_n P_\tau f  \in \mathcal{A}_\varepsilon$, where $h_n \in C^\infty_0(\bM)$ is an increasing, non negative,  sequence that converges to $1$. Letting then $\varepsilon \to 0$ , $ n \to \infty$ and $\tau \to 0$ proves that the inequality still holds for $f \in L^\infty(\bM)$.
\end{proof}

\subsection{Log-Harnack inequality}
An easy consequence of the Wang inequality of Proposition \ref{Wang inequality} is the following log-Harnack inequality.  
\begin{proposition}\label{log_harnack}
For $f \in L^\infty(\bM)$,   $\inf f  > 0$, $t>0$, $x,y \in \bM$,
$$
P_t(\ln f)(x) \leq \ln P_t(f)(y) + \left(\frac{ 1+\frac{2\kappa}{\rho_2}+ 2\rho^-_1 t }{4t}\right) d^2(x,y).
$$
\end{proposition}

The proof of this result appears in Section 2 of \cite{W3} where  a general study of these Harnack inequalities is done. For the sake of completeness, we reproduce the argument here.

\begin{proof} 
Applying Proposition \ref{Wang inequality} to the function $f^{\frac{1}{2^n}}$ for $\alpha= 2^n$, we get
$$
P_t\left(f^{2^{-n}}\right) (x) \leq \left( P_t(f) (y) \right)  ^{2^{-n}}  \exp \left( \frac{1}{2^n-1} \left(\frac{ 1+\frac{2\kappa}{\rho_2}+ 2\rho^-_1 t }{4t} \right) d^2(x,y)\right).
$$

Now, since $2^{-n}\to 0$ as $n\to  \infty$, by the dominated convergence theorem,
\begin{eqnarray*}
P_t ( \ln f )(x) &=& \lim_{n \to \infty} P_t \left(\frac{f^{2^{-n}} -1}{2^{-n}}\right) (x)\\
                   &\le & \lim_{n \to \infty} \left[  \frac{(P_t f(y))^{2^{-n}} \exp \left( \frac{1}{2^n-1} \left(\frac{ 1+\frac{2\kappa}{\rho_2}+ 2\rho^-_1 t }{4t} \right) d^2(x,y)\right) -1}{2^{-n}} \right]\\
                   &= & \lim_{n \to \infty} \left[  \frac{(P_t f(y))^{2^{-n}} -1}{2^{-n}} + (P_t f(y))^{2^{-n}} \frac{\exp \left( \frac{1}{2^n-1} \left(\frac{ 1+\frac{2\kappa}{\rho_2}+ 2\rho^-_1 t }{4t} \right) d^2(x,y)\right) -1}{2^{-n}} \right]\\
                   &=& \ln (P_t f)(y) + \left(\frac{ 1+\frac{2\kappa}{\rho_2}+ 2\rho^-_1 t }{4t} \right) d^2(x,y). 
\end{eqnarray*}
\end{proof}

When $\mu$ is a probability measure, the above log-Harnack inequalities implies the following lower bound for the heat kernel.
\begin{corollary}
Assume that $\mu$ is a probability measure, then for $t>0$, $x,y \in \bM$,
$$
p_{2t}(x,y) \geq \exp\left(-  \frac{ 1+\frac{2\kappa}{\rho_2}+ 2\rho^-_1 t }{4t} d^2(x,y)\right).
$$
\end{corollary}
\begin{proof}
Again, we reproduce an argument of  Wang \cite{W4}.
By applying Proposition \ref{log_harnack} to the function $f(\cdot)=p_t(x,\cdot)$ and integrating over the manifold, one gets:
$$
\int_\bM p_t(x,z) \ln p_t(x,z) d\mu(z) \leq \ln \int_\bM p_t( y,z) p_t(x,z) d\mu(z) +  \frac{ 1+\frac{2\kappa}{\rho_2}+ 2\rho^-_1 t }{4t} d^2(x,y).
$$
 
Now, by Jensen inequality, $\int_\bM p_t(x,z) \ln p_t(x,z) d\mu(z)  \geq 0$, thus 
$$
\ln p_{2t}(x,y) \geq -\frac{ 1+\frac{2\kappa}{\rho_2}+ 2\rho^-_1 t }{4t} d^2(x,y).
$$
 
\end{proof}

\subsection{Log-Sobolev inequality and exponential integrability of the square distance}

With Wang's inequality in hands, we can prove a log-Sobolev inequality provided the square integrability of the distance function. This extends a well-known theorem of Wang  (see \cite{SOB}, \cite{W2}).

\begin{theorem}\label{Wang}
Assume that the measure $\mu$ is a probability measure and that $L$ satisfies the generalized curvature dimension inequality \emph{CD}$(\rho_1,\rho_2,\kappa,\infty)$ for some $\rho_1 \in \mathbb{R} $,  $\rho_2 >0$, $\kappa \ge 0$.  Assume moreover that
\[
\int_\bM e^{\lambda d^2 (x_0,x )} d\mu(x) <+\infty,
\]
for some $x_0 \in \bM$ and $\lambda >  \frac{\rho_1^-}{2}$, then there is a constant $C>0$ such that for every function $f  \in C^ \infty_0(\bM)$,
\[
\int_\bM f^2 \ln f^2 d\mu -\int_\bM f^2 d\mu \ln  \int_\bM f^2 d\mu \le C \int_\bM \Gamma(f) d\mu.
\]
\end{theorem}

\begin{proof}

Let $\alpha >1$ and $f \in L^\infty(\bM)$, $ f \ge 0$. From Proposition \ref{Wang inequality}, by integrating with respect to $y$, we have
\begin{align*}
\int_\bM f^\alpha (y) d\mu (y) & \ge (P_t f)^\alpha (x)   \int_\bM \exp\left( -\frac{\alpha}{\alpha-1}  \left(\frac{ 1+\frac{2\kappa}{\rho_2}+ 2\rho^-_1 t }{4t}\right) d^2(x,y) \right) d\mu(y) \\
 &  \ge (P_t f)^\alpha (x)   \int_{B(x_0,1)} \exp\left( -\frac{\alpha}{\alpha-1}  \left(\frac{ 1+\frac{2\kappa}{\rho_2}+ 2\rho^-_1 t }{4t}\right) d^2(x,y) \right) d\mu(y) \\
 & \ge \mu( B(x_0,1)) (P_t f)^\alpha (x)  \exp\left( -\frac{\alpha}{\alpha-1}  \left(\frac{ 1+\frac{2\kappa}{\rho_2}+ 2\rho^-_1 t }{4t}\right) (d^2(x_0,x)+1) \right).
\end{align*}
As a consequence, we get
\[
(P_t f) (x) \le \frac{1}{ \mu( B(x_0,1))^\frac{1}{\alpha}}   \exp\left( \frac{1}{\alpha-1}  \left(\frac{ 1+\frac{2\kappa}{\rho_2}+ 2\rho^-_1 t }{4t}\right) (d^2(x_0,x)+1) \right) \| f \|_{L^\alpha} .
\]
Therefore if
\[
\int_\bM e^{\lambda d^2 (x_0,x )} d\mu(x) <+\infty,
\]
for some $x_0 \in \bM$ and $\lambda > \frac{\rho_1^-}{2}$, then we can find $1< \alpha <\beta$  and $t>0$ such that
\[
\| P_t f \|_{L^\beta} \le C_{\alpha,\beta} \| f \|_{L^\alpha} ,
\]
for some constant $C_{\alpha,\beta}$.  This implies the supercontractivity of the semigroup $(P_t)_{t \ge 0}$ and therefore from Gross' theorem (see \cite{bakry-stflour}), a defective logarithmic Sobolev inequality is satisfied, that is there exist two constants $A,B>0$  such that
\[
\int_\bM f^2 \ln f^2 d\mu -\int_\bM f^2 d\mu \ln  \int_\bM f^2 d\mu \le A \int_\bM \Gamma(f) d\mu+B\int_{\bM} f^2 d\mu, \quad f \in  C^ \infty_0(\bM).
\]
Now, since moreover the heat kernel is positive and the invariant measure a probability, we deduce from the uniform positivity improving property (see \cite{A}, Theorem 2.11) that $L$ admits a spectral gap.  That is, a Poincar\'e inequality is satisfied. It is then classical (see \cite{SOB}), that the conjunction of a spectral gap and a defective logarithmic Sobolev inequality implies the log-Sobolev inequality (i.e. we may actually take $B=0$ in the above inequality).  
\end{proof}

\subsection{A dimensional bound on the log-Sobolev constant}

If we take into account the dimension in the generalized curvature dimension inequality, we may obtain an upper bound for the log-Sobolev constant under the assumption that the curvature parameter $\rho_1$ is positive.

\begin{theorem}
Assume that the measure $\mu$ is a probability measure and that $L$ satisfies the generalized curvature dimension inequality \emph{CD}$(\rho_1,\rho_2,\kappa, d)$ for some $\rho_1 >0 $,  $\rho_2 >0$, $\kappa \ge 0$ and $d \ge 1$. For every function $f \in  C^ \infty_0(\bM)$,
\[
\int_\bM f^2 \ln f^2 d\mu -\int_\bM f^2 d\mu \ln  \int_\bM f^2 d\mu \le C  \int_\bM \Gamma(f) d\mu,
\]
with
\[
C=\frac{3(\rho_2+\kappa)}{\rho_1 \rho_2}  \left( 1+\Phi\left( \frac{d}{2} \left(1+\frac{3\kappa}{2\rho_2}\right)  \right) \right),
\]
where
\[
\Phi(x)=(1+x)\ln (1+x)-x \ln x.
\]
\end{theorem}

\begin{proof}
It is proved in \cite{BG1} that the  generalized curvature dimension inequality \emph{CD}$(\rho_1,\rho_2,\kappa, d)$ with $\rho_1 >0 $,  $\rho_2 >0$, $\kappa \ge 0$ and $d >0$ implies the following upper bound for the heat kernel: For $x,y \in \bM$ and $t>0$,
\[
p(x,y,t) \le \frac{1}{\left( 1-e^{-\frac{2\rho_1 \rho_2
t}{3(\rho_2+\kappa)}}
\right)^{\frac{d}{2}\left(1+\frac{3\kappa}{2\rho_2}\right)} }.
\]
Therefore, from Davies' theorem (Theorem 2.2.3 in \cite{Davies}), for $f \in  C^ \infty_0(\bM)$, we obtain the following defective log-Sobolev inequality which is valid for every $t>0$,
\[
\int_\bM f^2 \ln f^2 d\mu -\int_\bM f^2 d\mu \ln  \int_\bM f^2 d\mu \le 2t  \int_\bM \Gamma(f) d\mu -d\left(1+\frac{3\kappa}{2\rho_2}\right)   \ln \left( 1-e^{-\frac{2\rho_1 \rho_2
t}{3(\rho_2+\kappa)}} \right)\int_\bM f^2 d\mu.
\]
The previous heat kernel upper bound also implies that $-L$ has a spectral gap of size at least $\frac{2\rho_1 \rho_2}{3(\rho_2+\kappa)}$. Therefore, the following Poincar\'e inequality holds:
\[
\int_\bM f^2 d\mu-\left( \int_\bM f d\mu \right)^2\le \frac{3(\rho_2+\kappa)}{2\rho_1 \rho_2} \int_\bM \Gamma(f) d\mu.
\]
If we combine the two previous inequalities using Rothaus' inequality (see Lemma 4.3.8 in \cite{SOB}) and then chose the optimal $t$, we get the result.
\end{proof}

\begin{remark}
 It has been proved in \cite{BG1} that if $L$ satisfies the generalized curvature dimension inequality \emph{CD}$(\rho_1,\rho_2,\kappa, d)$ with $\rho_1 >0$, then $\bM$ needs to be compact.
\end{remark}

\subsection{Log-Sobolev inequality and diameter bounds}

We now generalize to a result due to Saloff-Coste \cite{SC1}.

\begin{theorem}
Assume that the measure $\mu$ is a probability measure and that $L$ satisfies the generalized curvature dimension inequality \emph{CD}$(\rho_1,\rho_2,\kappa, d)$ for some $\rho_1 \in \mathbb{R} $,  $\rho_2 >0$, $\kappa \ge 0$ and $d >0$. 

\begin{itemize}
\item The metric space $(\bM,d)$ is compact if and only if a log-Sobolev inequality 
\begin{align}\label{log-sob}
\int_\bM f^2 \ln f^2 d\mu -\int_\bM f^2 d\mu \ln  \int_\bM f^2 d\mu \le C  \int_\bM \Gamma(f) d\mu, \quad f \in C^\infty_0(\bM)
\end{align}
is satisfied for some $C>0$.
\item Moreover, if  $(\bM,d)$ is compact with diameter $D$ then, there is a constant $C(\rho_1,\rho_2,\kappa,d)$ such that
\[
D \le \frac{C(\rho_1,\rho_2,\kappa,d)}{\min (1, \rho_0)}
\]
where $\frac{2}{\rho_0}$ is the smallest constant $C$ such that (\ref{log-sob}) is satisfied.
\end{itemize}

\end{theorem}

\begin{proof}
If $\bM$ is compact, then
 \[
\int_\bM e^{\lambda d^2 (x_0,x )} d\mu(x) <+\infty,
\]
for every $x_0 \in \bM$ and $\lambda > \frac{ \rho_1^-}{2}$. Therefore, from Theorem \ref{Wang}, a log-Sobolev inequality is satisfied. 

Let us now assume that
\[
\int_\bM f^2 \ln f^2 d\mu -\int_\bM f^2 d\mu \ln  \int_\bM f^2 d\mu \le \frac{2}{\rho_0}  \int_\bM \Gamma(f) d\mu, \quad f \in C^\infty_0(\bM)
\]
is satisfied. 

Here we only sketch the proof,  since we may actually follow quite closely an argument from Ledoux \cite{L}  to which we refer for more details and in particular for tracking  the constants.
The key is to note that the curvature-dimension inequality \emph{CD}$(\rho_1,\rho_2,\kappa, d)$ for some $\rho_1 \in \mathbb{R} $,  $\rho_2 >0$, $\kappa \ge 0$ and $d >0 $ implies a Li-Yau type inequality (see Theorem 5.1 in \cite{BG1}). In particular  for $0<t \leq 1$ and a positive function $f$
$$
 0\leq A \frac{LP_t f}{P_t f} + \frac{B}{t}
$$
where $A$ and $B$ are some explicit positive constants depending only on $\rho_1, \rho_2, \kappa, d$.
Since $\frac{LP_t f}{P_t f}= \partial_t  \ln P_t f$, integrating between $t$ and $1$ yields, with $\gamma=\frac{B}{A}$,
$$
P_t f \leq \frac{1}{t^\gamma} P_1 f \textrm{ for all } 0<t\leq 1.
$$

Using now the equivalence between the log-Sobolev inequality and the hypercontractivity of the heat  semigroup due to Gross, we find that for $1<p<q<\infty$
$$
\Vert  P_t f\Vert_q \leq \Vert f\Vert_p
$$
as soon as $e^{2\rho_0 t} \geq \frac{q-1}{p-1}$.
Therefore, for $t=1$, $p=2$ and $q=1+e^{2\rho_0}$,
$$
\Vert  P_t f\Vert_q \leq \frac{1}{t^\gamma} \Vert  P_1 f\Vert_q \leq \frac{1}{t^\gamma} \Vert f\Vert _2 \textrm{ for } 0<t\leq 1.
$$
Such a semigroup estimate implies  a Sobolev inequality
$$
\Vert  f\Vert^2_r \leq 8( \Vert f\Vert^2_2 + \Vert \sqrt{\Ga(f,f)}\Vert_2^2)
$$ 
for some $r>2$ (see \cite{VSC}). Finally, the conjunction of the  logarithmic Sobolev inequality and of the above Sobolev inequality implies an entropy-energy inequality that may be used to prove  that the diameter is bounded (see \cite{bakry-stflour}).  Carefully tracking the constants leads to the desired bound for the diameter (see \cite{L}).
\end{proof}

\section{Log-Sobolev inequality and transportation cost inequalities}

In this section, we shall examine the links between the log-Sobolev inequality and some transportations cost inequalities. 
First, it is well known that the log-Sobolev inequality implies some transportation inequalities in a general "metric" setting. Conversely, on a weighted Riemannian manifold,  under the hypothesis that the Bakry-\'Emery curvature  is bounded from below, the converse implication holds true (see \cite{BGL}, \cite{OV}).

\

Therefore, in this Section we shall study  how some transportation inequalities can,  if the generalized curvature dimension inequality is satisfied, imply a log-Sobolev inequality. Unfortunately, we were only able to establish a partial converse in the sense that the  log-Sobolev inequality we obtain involves a term with $\Gamma^Z$.

\

\textbf{Throughout the Section, we assume $\mu(\bM)=1$}.

\

Let us  begin with some notations.
For  a positive function $f$ on $\bM$, we write
$$
\Ent_\mu (f)= \int_\bM f\ln f d\mu - \int_\bM f d\mu  \ln \int_\bM f d\mu.
$$
 We recall that the $L^2$-Wasserstein distance  of two measures $\nu_1$ and $\nu_2$ on $\bM$ is given by
 $$
 \W_2(\nu_1,\nu_2) ^2 = \inf_\Pi \int_\bM d^2(x,y) d\Pi (x,y)
$$
where the infimum is taken over all coupling of $\nu_1$ and $\nu_2$ that is on all probability measures $\Pi$ on $\bM\times \bM$ whose marginals are respectively $\nu_1$ and  $\nu_2$.

\subsection{A bound of the entropy of the semigroup by  the $L^2$-Wasserstein distance}

First we show how to bound the entropy of the the semigroup by this $L^2$-Wasserstein distance. The following result is a generalization in our setting of the Lemma 4.2 in \cite{BGL} (see also \cite{OV2} for an alternative proof which is more PDE oriented). 
 
\begin{proposition}\label{entropy_wasserstein}
Assume that $L$ satisfies the generalized curvature dimension inequality \emph{CD}$(\rho_1,\rho_2,\kappa,\infty)$ for some $\rho_1 \in \mathbb{R} $,  $\rho_2 >0$, $\kappa \ge 0$. Let $f$ be a non negative function on $\bM$ such that $\int_\bM f d\mu= 1$ and set $d\nu= fd\mu$. Then, for any $t>0$,
$$
\Ent_\mu( P_t f) \leq \left( \frac{ 1+\frac{2\kappa}{\rho_2}+ 2\rho^-_1 t }{4t} \right)\W_2(\mu,\nu)^2.
$$
\end{proposition}
\begin{proof}
Let $t>0$ and $f$ be a positive function on $\bM$ such that $\int_\bM f d\mu= 1$. The log-Harnack inequality of Proposition \ref{log_harnack} applied to the function $P_tf$ gives then:

$$
P_t(\ln P_tf)(x)\leq \ln P_{2t} (f)(y) + \frac{1}{s} d^2(x,y) .
$$
with 
$$ s= \frac{4t}{ 1+\frac{2\kappa}{\rho_2}+ 2\rho^-_1 t }.
$$
For $x$ fixed, by taking the infimum with respect to $y$ on the right hand side of the last inequality, we obtain
$$
P_t(\ln P_t f) (x) \leq  Q_s (\ln P_{2t}f) (x)
$$ 

where $Q_s$ is the infimum-convolution semigroup:
$$Q_s (\phi)(x)= \inf_{y\in \bM} \left\{ \phi(y)+ \frac{1}{2s} d(x,y)^2 \right\}.$$

Setting $\phi= \ln P_{2t}f$, by Jensen inequality
$$
\int_\bM \phi d\mu =\int_\bM \ln P_{2t} f d\mu \leq \ln \left(  \int_\bM P_{2t} f d\mu\right) =0,
$$
thus 
$$
P_t(\ln P_t f) (x) = Q_s (\phi) (x) -\int_\bM \phi d\mu.
$$ 
Since by symmetry:
$$
\Ent_\mu (P_t f) = \int_\bM f P_{t} (\ln P_t f) d\mu, 
$$
one finally gets:
$$
\Ent_\mu (f_t) \leq \sup_{\psi} \left\{ \int_\bM Q_s (\psi) (x) d\nu -\int_\bM \psi d\mu \right\}
$$
where the supremum is taken over all bounded mesurable functions  $\psi$ and where the measure $\nu$ is defined by $\frac{d\nu}{d\mu}= f$.  By Monge-Kantorovich duality, 

$$
\sup_{\psi} \left\{ Q_s (\psi) (x) -\int_\bM \psi d\mu \right\} = \inf_\Pi \int_\bM T(x,y) d\Pi (x,y)
$$
where the infimum is taken over all coupling of $\mu$ and $\nu$ 
 and where the cost $T$ is just 
$$
T(x,y)= \frac{1}{s} d^2(x,y).
$$
Therefore the latter infimum is equal to $\frac{1}{s} W_2(\mu,\nu)^2$.
\end{proof}

\subsection{Modified HWI inequality}

The following Lemma may be proved in the very same way as Proposition \ref{log-bound}.

\begin{lemma}\label{gb2}
Assume that $L$ satisfies the generalized curvature dimension inequality \emph{CD}$(\rho_1,\rho_2,\kappa,\infty)$ for some $\rho_1 \in \mathbb{R} $,  $\rho_2 >0$, $\kappa \ge 0$. Let $\varepsilon >0$ and $f \in \mathcal A_\varepsilon$. 
 For $x\in \M$,  $t \ge 0$ one has:
 \[
 P_t f \Gamma( \ln P_t f) +P_tf  \Gamma^Z(\ln P_t f) \le e^{2\alpha t} \left(  P_t ( f \Gamma(\ln f)) + P_t ( f \Gamma^Z(\ln f))\right), \quad t \ge 0,
 \]
 where $\alpha=- \min (\rho_2, \rho_1-\kappa, 0)$.
\end{lemma}

We may now prove:

\begin{theorem}
Assume that $L$ satisfies the generalized curvature dimension inequality \emph{CD}$(\rho_1,\rho_2,\kappa,\infty)$ for some $\rho_1 \in \mathbb{R} $,  $\rho_2 >0$, $\kappa \ge 0$. If the quadratic transportation cost inequality 
\begin{equation}
\W_2 (\mu, \nu)^2 \leq c \Ent_\mu \left( \frac{d\nu}{d\mu}\right) 
\end{equation} 
 is satisfied for every absolutely continuous probability measure $\nu$  with a constant $c<\frac{2}{\rho_1^-}$, then the following modified log-Sobolev inequality 
$$\Ent_\mu(f) \leq C_1 \int_\bM \frac{\Ga(f)}{f} d\mu +C_2\int_\bM \frac{\Ga^Z(f)}{f} d\mu, \quad f \in \mathcal{A}_\varepsilon, \varepsilon >0, $$
holds for some constants $C_1$ and $C_2$ depending only on $c ,\rho_1, \kappa,\rho_2$.
\end{theorem}

\begin{proof}

Let $f \in \mathcal{A}_\varepsilon$ such that $\int_\bM f d\mu=1$, by the diffusion property, we have 
$$
\frac{d}{dt} \Ent_\mu(P_t f) =  - I(P_t f) 
$$

with
 $$
I(P_t f) = \int_\bM \frac{\Ga(P_t f)}{P_t f} d\mu.
$$
From Lemma \ref{gb2}, we have
\[
\frac{\Ga(P_t f)}{P_t f}  \le  e^{2\alpha t} \left(  P_t ( f \Gamma(\ln f)) + P_t ( f \Gamma^Z(\ln f))\right),
\]
which implies, by integration over the manifold $\bM$,
\[
I(P_t f) \le e^{2\alpha t} \left( \int_\bM  \frac{ \Gamma( f)}{f} d\mu + \int_\bM  \frac{ \Gamma^Z( f)}{f} d\mu\right).
\]
As a consequence,
\begin{eqnarray*}
\Ent_\mu (f)& \leq & \int_0^T I(P_t f) dt + \Ent_\mu (P_T f) \\
            &\leq &  \left(  \int_0^T e^{2\alpha t} dt\right)  \left( \int_\bM  \frac{ \Gamma( f)}{f} d\mu + \int_\bM  \frac{ \Gamma^Z( f)}{f} d\mu\right)+ \Ent_\mu(P_T f).
\end{eqnarray*}
We now use Proposition \ref{entropy_wasserstein} and infer
\[
\Ent_\mu (f) \le  \left(  \int_0^T e^{2\alpha t} dt\right)  \left( \int_\bM  \frac{ \Gamma( f)}{f} d\mu + \int_\bM  \frac{ \Gamma^Z( f)}{f} d\mu\right)+ \left( \frac{ 1+\frac{2\kappa}{\rho_2}+ 2\rho^-_1 T }{4T} \right)\W_2(\mu,\nu)^2,
\]
where $d\nu =f d\mu$. Using the assumption $\W_2 (\mu, \nu)^2 \leq c \Ent_\mu \left( f \right) $ and chosing $T$ big enough finishes the proof.
\end{proof}

\section{Logarithmic isoperimetric inequality}

In this section, we assume that the measure $\mu$ is a probability measure, that is $\mu(\bM)=1$, and we show how the curvature dimension inequality \emph{CD}$(\rho_1,\rho_2,\kappa,\infty)$ together with a log-Sobolev inequality implies a logarithmic isoperimetric inequality of Gaussian type. The method used here is very close from the one in Ledoux \cite{L2}. 

We first  need to precise what we mean by the perimeter of a  set  in our subelliptic setting: This is essentially done in \cite{GN}.

Let us first observe that, given any point $x\in \M$ there exists an open set $x\in U\subset \M$ in which the operator $L$ can be written as 
\begin{equation}\label{locrep}
L = - \sum_{i=1}^m X^*_i X_i,
\end{equation}
where  the vector fields $X_i$ have Lipschitz continuous coefficients in $U$, and $X_i^*$ indicates the formal adjoint of $X_i$ in $L^2(\M,d\mu)$.

We indicate with $\mathcal F(\bM )$ the set of $C^1$ vector fields which are subunit for $L$. Given a function $f\in
L^1_{loc}(\bM)$, which is supported in $U$ we define the horizontal total variation of $f$ as
\[
\text{Var}(f) = \underset{\phi\in \F(\bM)}{\sup} \int_U
f \left(\sum_{i=1}^m X^*_i \phi_i\right) d\mu,
\]
where on $U$, $\phi=\sum_{i=1}^m \phi_i X_i$. For functions not supported in $U$, $\text{Var}(f) $ may be defined by using a partition of unity. 
The space \[ BV (\bM) = \{f\in L^1(\bM)\mid
\text{Var}(f)<\infty\},
\]
endowed with the norm
\[
||f||_{BV(\bM)} = ||f||_{L^1(\bM)} + \text{Var}(f),
\]
is a Banach space. It is well-known that $W^{1,1}(\bM) = \{f\in
L^1(\bM)\mid \sqrt{\Ga f}\in L^1(\bM)\}$ is a strict subspace of
$BV(\bM)$ and when $f\in
W^{1,1}(\bM)$ one has in fact
\[
\text{Var}(f) = ||\sqrt{\Gamma(f)}||_{L^1(\bM)}.
\]
Given a measurable set $E\subset \bM$ we say that it has finite perimeter  if $\mathbf 1_E\in BV(\bM)$. In
such case the perimeter of $E$ is by
definition
\[
P(E) = \text{Var}(\mathbf 1_E).
\]
We will need the following approximation result, see Theorem 1.14 in \cite{GN}.

\begin{lemma}\label{P:ag}
Let $f\in BV(\bM)$, then there exists a sequence $\{f_n\}_{n\in
\mathbb N}$ of functions in $C_0^\infty(\bM)$ such that:
\begin{itemize}
\item[(i)] $||f_n - f||_{L^1(\bM)} \to 0$;
\item[(ii)] $\int_\bM \sqrt{\Gamma(f_n)} d\mu \to
\text{Var}(f)$.
\end{itemize}
\end{lemma}

After this digression, we now state the main theorem of this Section.

\begin{theorem} 

Assume that $L$ satisfies the generalized curvature dimension inequality \emph{CD}$(\rho_1,\rho_2,\kappa,\infty)$ and that $\mu $ satisfies the log-Sobolev inequality:
\begin{equation}
\int_\bM f^2 \ln f^2 d\mu -\int_\bM f^2 d\mu \ln  \int_\bM f^2 d\mu \le \frac{2}{\rho_0}  \int_\bM \Gamma(f) d\mu, 
\end{equation} for all smooth functions $f \in C_0^\infty(\bM)$.
Let $A$ be a set of the  manifold $\M$  which has a finite perimeter $P(A)$ and such that $0\leq \mu(A)\leq \frac{1}{2}$, then 
$$
P( A)\geq \frac{\ln 2 }{ 4\left(3+\frac{2\kappa}{\rho_2}\right)  } \min \left(\sqrt \rho_0 ,\frac{\rho_0} {\sqrt {\rho_1^-}}\right)   \mu(A)\left(\ln \frac{1}{\mu(A)} \right)^\frac{1}{2}.
$$
\end{theorem}

\begin{remark}
The constant $\frac{\ln 2 }{ 4\left(3+\frac{2\kappa}{\rho_2}\right)  }$ is of course not optimal but unlike the result stated in  \cite{L2}, it does not depend on the dimension. This comes from the fact that we use the reverse Poincar\'e inequality of Proposition \ref{rev-poinc}  instead of a Li-Yau type inequality to obtain the Lemma \ref{reg-bound-L1} below.
\end{remark}
\begin{lemma}\label{reg-bound-L1}
Assume that $L$ satisfies the generalized curvature dimension inequality \emph{CD}$(\rho_1,\rho_2,\kappa,\infty)$, let $f \in C^\infty_0(\bM)$, then for all $t>0$
\begin{equation}
\Vert f-P_t f \Vert_1 \leq \left( \frac{1}{2}+\frac{\kappa}{\rho_2}  + \rho_1^- t \right) \sqrt t    \Vert\sqrt{\Ga( f)} \Vert_1.
\end{equation}
\end{lemma}

\begin{proof}
First, since the curvature dimension inequality \emph{CD}$(\rho_1,\rho_2,\kappa,\infty)$ holds true, by Corollary \ref{reg-bound}, 
for all $g \in C_0^\infty(\bM)$ and for all $0<t\leq t_0$,
$$
\Vert\sqrt{\Ga(P_t g)} \Vert_\infty  \leq  \left( \frac{\frac{1}{2}+\frac{\kappa}{\rho_2}+  \rho^-_1 t_0}{t} \right)^\frac{1}{2} \Vert g\Vert_\infty.
$$

Therefore, by duality, for every positive and smooth function $f$, every smooth function $g$ such that  $\Vert g \Vert_\infty\leq 1$ and all $0<t\leq t_0$,

\beas
\int_\bM g(f-P_t f)d\mu &=& -\int_0^t \int_\M g L P_s f d\mu ds \\
                    &=& \int_0^t \int_\M \Ga(P_s g, f) d\mu ds\\
                    &\leq & \Vert\sqrt{\Ga( f)} \Vert_1 \int_0^t \Vert\sqrt{\Ga(P_s g)} \Vert_\infty ds \\
                    &\leq & \left( \frac{1}{2}+\frac{\kappa}{\rho_2}  + \rho_1^- t_0 \right) \sqrt t    \Vert\sqrt{\Ga( f)} \Vert_1
\eeas

 which ends the proof.
\end{proof}

Now we can turn to the proof of the Theorem.

\begin{proof}
Let $A$ be a set with finite perimeter.
Applying Lemma \ref{reg-bound-L1} to smooth functions approximating the characteristic function $\1_A$ as in Lemma \ref{P:ag} gives 
$$
\Vert \1_A - P_t \1_A \Vert_1 \leq \left( \frac{1}{2}+\frac{\kappa}{\rho_2}  + \rho_1^- t \right) \sqrt t P( A).
$$
By symmetry and stochastic completness of the semigroup, 
\beas
\Vert \1_A - P_t \1_A \Vert_1 &=& \int_A (1-P_t \1_A) d\mu + \int_{A^c} P_t(\1_A) d\mu\\
                              &=& \int_A (1-P_t \1_A) d\mu + \int_A (P_t \1_{A^c}) d\mu\\
                              &=& 2 \left( \mu(A)- \int_A P_t (\1_A) d\mu \right)\\
                              &=& 2  \left( \mu(A)- \Vert P_\frac{t}{2} (\1_A) \Vert_2^2 \right).
\eeas

Now we can use the hypercontracivity constant to bound $\Vert P_\frac{t}{2} (\1_A) \Vert_2^2$. Indeed, from Gross' theorem it is well known that that the logarithmic Sobolev inequality  
\[
\int_\bM f^2 \ln f^2 d\mu -\int_\bM f^2 d\mu \ln  \int_\bM f^2 d\mu \le \frac{2}{\rho_0}  \int_\bM \Gamma(f) d\mu, \quad f \in C^\infty_0(\bM),
\]
is equivalent to hypercontracivity property
$$
\Vert P_t f \Vert_q \leq \Vert f \Vert_p 
$$
 for all $f$ in $L^p(\M)$ whenever $1<p<q<\infty$ and $e^{\rho_0 t}\geq \sqrt \frac{q-1}{p-1}$.

Therefore, with $p(t)= 1+e^{-\rho_0 t}<2$, we get,
\beas
\left( \frac{1}{2}+\frac{\kappa}{\rho_2}  + \rho_1^- t \right) \sqrt t P(A) &\geq & 2  \left( \mu(A)- \mu(A)^\frac{2}{p(t)} \right)\\
                                                                                     &\geq & 2 \mu(A) \left(1- \mu(A)^\frac{1-e^{-\rho_0 t}}{1+e^{-\rho_0 t}} \right).
\eeas
Since for $x>0$
$$
1-e^{-x} \geq \min\left(\frac{x}{2},\frac{1}{2}\right) \; \textrm{ and } \; \frac{1-e^{-x}}{1+e^{-x}}\geq \min \left(\frac{x}{4}, \frac{1}{2}\right),
$$

$$
\mu(A)^\frac{1-e^{-\rho_0 t}}{1+e^{-\rho_0 t}}\leq \exp\left( - \min \left(\frac{\rho_0 t}{4}, \frac{1}{2}\right) \ln \frac{1}{\mu(A)}\right)
$$

$$
1-\mu(A)^\frac{1-e^{-\rho_0 t}}{1+e^{-\rho_0 t}} \geq \min\left( \min \left(\frac{\rho_0 t}{8}, \frac{1}{4}\right)\ln\frac{1}{\mu(A)}, \frac{1}{2} \right).
$$
Therefore for all $t>0$,
\begin{equation}\label{isop}
P( A) \geq \frac{2}{\left(\frac{1}{2}+\frac{\kappa}{\rho_2}  + \rho_1^- t\right) \sqrt t} \mu(A) \min\left( \min \left(\frac{\rho_0 t}{8}, \frac{1}{4}\right)\ln\frac{1}{\mu(A)}, \frac{1}{2} \right).
\end{equation}
With $t_0=\min\left(\frac{1}{\rho_0},\frac{1}{\rho_1^-}\right)$, for $0<t\leq t_0$, we have
$$
P( A) \geq \frac{2}{\left(\frac{1}{2}+\frac{\kappa}{\rho_2}  + 1\right) \sqrt t} \mu(A) \min\left(  \frac{\rho_0 t}{8} \ln\frac{1}{\mu(A)}, \frac{1}{2} \right).
$$Now, if $\mu(A)$ is small enough, i.e. $\mu(A)\leq e^{-4}$, we can chose $t=\frac{4t_0}{ \ln \frac{1}{\mu(A)}}\leq t_0$ so that $\min\left(  \frac{\rho_0 t}{8} \ln\frac{1}{\mu(A)}, \frac{1}{2} \right)=\frac{\rho_0 t_0}{2}$ and then  get
$$
P( A)\geq \frac{ \rho_0 \sqrt {t_0}   \mu(A)\left(\ln \frac{1}{\mu(A)} \right)^\frac{1}{2} }{3+\frac{2\kappa}{\rho_2}  }
$$  

When $0\leq \mu(A)\leq \frac{1}{2}$, we can apply (\ref{isop}) with $t=t_0$ and since $\ln \frac{1}{\mu(A)} \geq \ln 2$,
$$
\min\left( \frac{\rho_0 t_0}{8}\ln\frac{1}{\mu(A)}, \frac{1}{2} \right) \geq \frac{\rho_0 t_0 \ln 2}{8}
$$
and thus
$$
P( A)\geq \frac{\ln 2 \rho_0 \sqrt {t_0}   \mu(A) }{2\left(3+\frac{2\kappa}{\rho_2}\right) }.
$$

Noticing $\ln\frac{1}{\mu(A)}\leq 4$ if $\mu(A)\geq e^{-4}$, we obtain that for every $A$ with $0\leq \mu(A)\leq\frac{1}{2}$,
$$
P( A)\geq \frac{ \rho_0 \sqrt {t_0}   \mu(A)\left(\ln \frac{1}{\mu(A)} \right)^\frac{1}{2}\ln 2 }{ 4\left(3+\frac{2\kappa}{\rho_2}\right)  }.
$$
Keeping in mind that $t_0=\min\left(\frac{1}{\rho_0},\frac{1}{\rho_1^-}\right)$ ends the proof.
\end{proof}

\end{document}